\documentclass{article}

\usepackage{arxiv}

\usepackage[utf8]{inputenc} 
\usepackage[T1]{fontenc}    
\usepackage{hyperref}       
\usepackage{url}            
\usepackage{booktabs}       
\usepackage{amsfonts}       
\usepackage{nicefrac}       
\usepackage{microtype}      
\usepackage{lipsum}

\usepackage{morefloats}
\usepackage{color}
\usepackage{multirow}
\usepackage{latexsym}
\usepackage{amsmath,amssymb,amsthm,graphicx}
\usepackage{dsfont}
\usepackage{enumitem}
\usepackage{hyperref}
\usepackage{arydshln}

\newtheoremstyle{thm}
{9pt}
{9pt}
{\itshape}
{}
{\bfseries}
{.}
{ }
{}
\theoremstyle{thm}

\newtheorem{theorem}{Theorem}[section]

\newtheorem{corollary}[theorem]{Corollary}

\newcommand{\vertk}{\stackrel{{\cal D}}{\longrightarrow}}
\newcommand{\fsk}{\stackrel{{\rm a.s.}}{\longrightarrow}}

\newtheoremstyle{def}
{9pt}
{9pt}
{}
{}
{\bfseries}
{.}
{ }
{}
\theoremstyle{def}

\newcommand{\R}{\mathbb{R}} 

\renewcommand{\footnoterule}{%
	\kern -3.5pt
	\hrule width \textwidth height 1pt
	\kern 3.5pt
}

\makeatletter
\def\blfootnote{\xdef\@thefnmark{}\@footnotetext}
\makeatother

\title{A new test of multivariate normality by a double estimation in a characterizing PDE}

\author{Philip D\"orr\\
Institute of Stochastics, \\
Karlsruhe Institute of Technology (KIT), \\
Englerstr. 2, D-76133 Karlsruhe. \\
\And  Bruno Ebner\\
 Institute of Stochastics, \\
Karlsruhe Institute of Technology (KIT), \\
Englerstr. 2, D-76133 Karlsruhe. \\
\texttt{Bruno.Ebner@kit.edu}\\
\And
Norbert Henze\\
Institute of Stochastics, \\
Karlsruhe Institute of Technology (KIT), \\
Englerstr. 2, D-76133 Karlsruhe. \\
\texttt{Norbert.Henze@kit.edu}\\
}

\begin{document}

\date{\today}
\maketitle


\begin{abstract}
This paper deals with testing for nondegenerate normality of a $d$-variate random vector $X$ based on
a random sample $X_1,\ldots,X_n$ of $X$. The rationale of the test is that the characteristic function $\psi(t) = \exp(-\|t\|^2/2)$ of the standard normal distribution in $\mathbb{R}^d$ is the only solution of the partial differential equation $\Delta f(t) = (\|t\|^2-d)f(t)$, $t \in \mathbb{R}^d$, subject to the condition $f(0) = 1$. By contrast with a recent approach that bases a test for multivariate normality on the
difference $\Delta \psi_n(t)-(\|t\|^2-d)\psi(t)$, where $\psi_n(t)$ is the empirical
characteristic function of suitably scaled residuals of $X_1,\ldots,X_n$, we consider a
weighted $L^2$-statistic that employs $\Delta \psi_n(t)-(\|t\|^2-d)\psi_n(t)$. We derive asymptotic properties of the test under the null hypothesis and alternatives. The test is affine invariant and consistent against general alternatives, and it exhibits high power when compared with prominent competitors.
\end{abstract}

\section{Introduction}\label{sec:Intro}
A useful tool for assessing the fit of data to a family of distributions are empirical counterparts of distributional characterizations.
Such characterizations often emerge as solutions of an equation of the type $\rho({\rm D}f,f)=0$.
Here, $f$ may be the moment generating function, the Laplace transform, or the characteristic function, and D denotes a differential operator, i.e., this operator can be regarded as ordinary differentiation if $f$ is a function of only one variable or, for instance, the Laplace operator in the multivariate case. Such (partial) differential equations have been used to test for multivariate normality,
see \cite{DEH:2019,HV:2019}, exponentiality, see \cite{BH:1991}, the gamma distribution, see \cite{HME:2012}, the inverse Gaussian distribution,
see \cite{HK:2002}, the beta distribution, see \cite{RM:2018}, the univariate and multivariate skew-normal distribution, see \cite{M:2010} and \cite{MH:2010}, and the Rayleigh distribution, see \cite{MI:2003}. In all these references, the authors propose a goodness-of-fit test by plugging in an empirical counterpart $f_n$ for $f$ into $\rho({\rm D}f,f)$, and by measuring the deviation from the zero function in a suitable function space. If, under the hypothesis to be tested, the function $f$ has a closed form and is known, there are two options for obtaining an empirical counterpart to the characterizing equation, namely $\rho({\rm D}f_n,f)=0$, or $\rho({\rm D}f_n,f_n)=0$. To the best of our knowledge, the effect of considering both options for the {\it same testing problem} and to study the consequences on the performance of the resulting test statistics has not yet been considered, neither from a theoretical point of view, nor in a simulation study. In this spirit, the purpose of this paper is to investigate the effect on the power of a recent test for multivariate normality based on a characterization of the multivariate normal law in connection with the harmonic oscillator, see \cite{DEH:2019}.

In what follows, let $d\ge 1$ be a fixed integer, and let $X,X_1,\ldots,X_n,\ldots$ be independent and identically distributed (i.i.d.)
 $d$-dimensional random (column) vectors, that are defined on a common probability space $(\Omega,{\cal A},\mathbb{P})$.
We write $\mathbb{P}^X$ for the distribution of $X$, and we denote the $d$-variate normal law with expectation $\mu$
and nonsingular covariance matrix $\Sigma$ by N$_d(\mu,\Sigma)$. Moreover,
$\mathcal{N}_d=\{{\rm N}_d(\mu,\Sigma):\mu\in\mathbb{R}^d,\,\Sigma\in\mathbb{R}^{d\times d}\; \mbox{positive definite}\}
$
stands for the class of all nondegenerate $d$-variate normal distributions. To check the assumption of multivariate normality means to test the hypothesis
\begin{equation}\label{H0}
H_0:\,\mathbb{P}^X\in\mathcal{N}_d,
\end{equation}
against general alternatives. The starting point of this paper is Theorem 1.1 of \cite{DEH:2019}. To state this result,
let $\Delta$ denote the Laplace operator,  $\|\cdot\|$ the Euclidean norm in $\mathbb{R}^d$, and  I$_d$  the identity matrix of size $d$.
Then Theorem  1.1 of \cite{DEH:2019} states that the characteristic function $\psi(t)=\exp\left(-\|t\|^2/2\right)$, $t\in\R^d,$ of the $d$-variate
standard normal distribution N$_d(0,{\rm I}_d)$  is the unique solution of the partial differential equation
\begin{equation}\label{PDE}
\left\{\begin{array}{l}\Delta f(x) - (\|x\|^2-d) f(x)=0,\quad x\in\R^d,\\ f(0)=1.\end{array}\right.
\end{equation}
Writing $\overline{X}_n= n^{-1} \sum_{j=1}^nX_j$ for the sample mean and $S_n= n^{-1} \sum_{j=1}^n(X_j-\overline{X}_n)(X_j-\overline{X}_n)^\top$
 for the sample covariance matrix of $X_1,\ldots,X_n$, respectively, where the superscript $\top$ means transposition, the standing tacit assumptions
 that $\mathbb{P}^X$ is absolutely continuous with respect to Lebesgue measure and $n \ge d+1$ guarantee that $S_n$ is invertible almost surely, see \cite{EP:1973}.
The test statistic is based on the so-called {\em scaled residuals}
\[
Y_{n,j}=S_n^{-1/2}(X_j-\overline{X}_n), \quad j =1,\ldots,n.
\]
Here, $S_n^{-1/2}$ is the unique symmetric
positive definite square root of $S_n^{-1}$.
Letting $\psi_n(t)=n^{-1}\sum_{j=1}^n\exp({\rm{i}}t^\top Y_{n,j})$, $t\in\R^d$, denote the empirical characteristic function (ecf) of $Y_{n,1},\ldots,Y_{n,n}$,
the test statistic proposed in \cite{DEH:2019} is
\begin{equation}\label{Tna}
T_{n,a} = n\int_{\mathbb{R}^d}\left|\Delta\psi_n(t)-\Delta \psi(t) \right|^2w_a(t)\, \mbox{d}t,
\end{equation}
where
\begin{equation}\label{defwa}
w_a(t)= \exp(-a\|t\|^2), \quad t \in \R^d,
\end{equation}
and $a>0$ is a fixed constant. The statistic $T_{n,a}$ has a nice closed-form expression as a function
of $Y_{n,i}^\top Y_{n,j}$, $i,j \in \{1,\ldots,n\}$ (see display (10)-(12) of \cite{DEH:2019}) and is thus
invariant with respect to full-rank affine transformations of $X_1,\ldots,X_n$. Theorems 2.2 and 2.3 of
\cite{DEH:2019} show that, elementwise on the underlying probability space, suitably rescaled versions of
$T_{n,a}$ have limits as $a \to \infty$ and $a \to 0$, respectively. In the former case, the limit
is a measure of multivariate skewness, introduced in     \cite{MRS:1993}, whereas Mardia's time-honored
measure of multivariate kurtosis (see \cite{MAR:1970}) shows up as $a \to 0$. As $n \to \infty$,
the statistic $T_{n,a}$ has a nondegenerate limit null distribution (Theorem 4.1 of \cite{DEH:2019}),
and a test of \eqref{H0} that rejects $H_0$ for large values of $T_{n,a}$ is able to detect
alternatives that approach $H_0$ at the rate $n^{-1/2}$, irrespective of the dimension $d$ (Corollary 5.2 of
\cite{DEH:2019}). Under an alternative distribution satisfying $\mathbb{E}\|X\|^4 < \infty$, $n^{-1}T_{n,a}$ converges
almost surely to a measure of distance $\Delta_a$ between $\mathbb{P}^X$ and the class ${\cal N}_d$ (Theorem 6.1 of \cite{DEH:2019}).
As a consequence, the test for multinormality based on $T_{n,a}$ is consistent against any such alternative.
By Theorem 6.5 of \cite{DEH:2019}, the sequence $\sqrt{n}(n^{-1}T_{n,a}-\Delta_a)$ converges in distribution to a centered
normal law. Since the variance of this limit distribution can be estimated consistently from $X_1,\ldots,X_n$ (Theorem 6.7
of \cite{DEH:2019}), we have an asymptotic confidence interval for $\Delta_a$.

The novel approach taken in this paper is to replace {\em both} of the functions $f$ occurring in   \eqref{PDE} by the
ecf $\psi_n$. Since, under $H_0$, $\Delta \psi_n(t)$ and $(\|t\|^2 -d)\psi_n(t)$ should be close to each other for large $n$,
it is tempting to see what happens if, instead of $T_{n,a}$ defined in   \eqref{Tna}, we base a test of $H_0$ on the weighted $L^2$-statistic
\begin{equation}\label{defuna}
U_{n,a} = n\int_{\R^d}\left|\Delta\psi_n(t)-\left(\|t\|^2-d\right)\psi_n(t)\right|^2w_a(t)\, \mbox{d}t
\end{equation}
and reject $H_0$ for large values of $U_{n,a}$.

Since $\Delta\psi_n(t)=- n^{-1} \sum_{j=1}^n\|Y_{n,j}\|^2\exp({\rm{i}}t^\top Y_{n,j})$, the relation
\begin{eqnarray}\label{imprel}
     &&\int_{\R^d} (\|t\|^2 - d)^2\cos(t^{\top}c) \exp(-a\|t\|^2) {\rm d}t \\ \nonumber
    &=& \left(\frac{\pi}{a}\right)^{d/2}\frac{16d^2a^3(a\! -\! 1) + 4d(d\! +\! 2)a^2 + (8da^2\! -\!  4(d\! +\! 2)a)\|c\|^2\! +\!  \|c\|^4}{16a^4}\exp\left(-\frac{\|c\|^2}{4a}\right),
\end{eqnarray}
valid for $c \in \mathbb{R}^d$ and $a>0$,  and tedious but straightforward calculations yield the representation
\begin{align} \label{repruna}
    U_{n,a} =~& \left(\frac{\pi}{a}\right)^{d/2}\frac{1}{n}\sum_{j,k=1}^n \Biggl[\|Y_{n,j}\|^2\|Y_{n,k}\|^2\exp\left(-\frac{\|Y_{n,j}\! -\!  Y_{n,k}\|^2}{4a}\right) \\ \nonumber
    &-(\|Y_{n,j}\|^2\! + \!  \|Y_{n,k}\|^2)\frac{1}{4a^2}\bigl(\|Y_{n,j}\! -\!  Y_{n,k}\|^2 + 2ad(2a\! -\! 1)\bigr)\exp\left(-\frac{\|Y_{n,j}\! -\! Y_{n,k}\|^2}{4a}\right) \\  \nonumber
    &+\frac{1}{16a^4}\exp\left(-\frac{\|Y_{n,j}\! - \!  Y_{n,k}\|^2}{4a}\right)\Bigl(16d^2a^3(a\! -\! 1) + 4d(d\! +\! 2)a^2  + \|Y_{n,j}\! - \! Y_{n,k}\|^4 \\ \nonumber
    &\hspace{5.7cm}+ \bigl(8da^2\! -\!  4(d\! +\! 2)a\bigr)\|Y_{n,j}\! - \!  Y_{n,k}\|^2\Bigr)\Biggr],
\end{align}
which is amenable to computational purposes. Moreover, $U_{n,a}$ turns out to be affine invariant.

The rest of the paper is organized as follows. In Section \ref{secinfty}, we derive the elementwise limits of $U_{n,a}$,
 after suitable transformations, as $a \to 0$ and $a \to \infty$.   Section \ref{seclimitnuuld} deals with the
 limit null distribution of $U_{n,a}$ as $n \to \infty$.  In Section \ref{seclimitsalt}, we show that, under the condition $\mathbb{E}\|X\|^4 < \infty$, $n^{-1}U_{n,a}$ has an almost sure limit as $n \to \infty$ under a fixed alternative to normality. As a consequence, the test based on $U_{n,a}$ is consistent against any such alternative. Moreover, we prove that the asymptotic distribution of $U_{n,a}$, after a suitable transformation, is  a centered normal distribution. In Section \ref{secsimul}, we present the results of a simulation study that compares the power of the test for normality based on $U_{n,a}$ with that of prominent competitors. Section \ref{secdata} shows a real data example, and Section \ref{secconclus} contains some conclusions and gives an outlook on potential further work.

\section{The limits $a\to 0$ and $a \to \infty$}\label{secinfty}
This section considers the (elementwise) limits of $U_{n,a}$ as $a \to 0$ and $a \to \infty$. The results shed some light on the role
of the parameter $a$ that figures in the weight function $w_a$ in \eqref{defwa}.  Notice that, from the definition of $U_{n,a}$ given in
\eqref{defuna}, we have $\lim_{a\to \infty} U_{n,a} =0$ and $\lim_{a\to 0} U_{n,a} =\infty$, since
$\int \left|\Delta\psi_n(t)-\left(\|t\|^2-d\right)\psi_n(t)\right|^2\, \mbox{d}t=\infty$. Suitable transformations of $U_{n,a}$, however,
yield well-known limit statistics as $a \to 0$ and $a \to \infty$.

\begin{theorem} \label{thmatozero}
Elementwise on the underlying probability space, we have
\begin{equation}\label{limesato0}
\lim_{a\rightarrow 0} \left[\left(\frac{a}{\pi}\right)^{d/2}U_{n,a} - \frac{d(d+2)}{4a^2}\right] = \frac{1}{n}\sum_{j=1}^n \|Y_{n,j}\|^4 - d^2.
\end{equation}
\end{theorem}

\begin{proof}
Starting with \eqref{repruna}, $(a/\pi)^{d/2} U_{n,a}$ is, apart from the factor $1/n$, a double sum over $j$ and $k$. Since each summand for which $j\ne k$ vanishes
asymptotically as $a \to 0$, we have
\[
\left(\frac{a}{\pi}\right)^{d/2}U_{n,a}  =  \frac{1}{n} \sum_{j=1}^n \bigg{[} \|Y_{n,j}\|^4 - \frac{d(2a-1)}{a} \, \|Y_{n,j}\|^2 + \frac{d^2(a-1)}{a} + \frac{d(d+2)}{4a^2}\bigg{]} + o(1)
\]
as $a \to 0$, and the result follows from the fact that $\sum_{j=1}^n \|Y_{n,j}\|^2 = nd$.
\end{proof}

Theorem \ref{thmatozero} means that a suitable affine transformation of $U_{n,a}$ has a limit as $a \to 0$, and that this limit is -- apart from the additive constant $d^2$ -- the
time-honored measure of multivariate kurtosis in the sense of Mardia, see \cite{MAR:1970}. The same measure -- without the subtrahend $d^2$ -- shows up as a limit of
$(a/\pi)^{d/2} T_{n,a}$ as $a \to 0$, see Theorem 2.3 of \cite{DEH:2019}. The next result shows that $U_{n,a}$ and $T_{n,a}$, after multiplication with the same
scaling factor, converge to the same limit as $a \to \infty$, cf. Theorem 2.1 of \cite{DEH:2019}.

\begin{theorem} \label{thmatoinfty}
Elementwise on the underlying probability space, we have
\begin{equation}\label{limitmrz}
\lim_{a\rightarrow\infty} \frac{2}{n\pi^{d/2}}a^{d/2+1}U_{n,a} = \frac{1}{n^2} \sum_{j,k=1}^n \|Y_{n,j}\|^2 \|Y_{n,k}\|^2 Y_{n,j}^\top Y_{n ,k}.
\end{equation}
\end{theorem}

\begin{proof}
The proof follows the lines of the proof of Theorem 2.1 of \cite{DEH:2019} and is thus omitted. \end{proof}

The limit figuring on the right hand side  of \eqref{limitmrz} is a measure of multivariate skewness, introduced by
M\'{o}ri, Rohatgi and Sz\'{e}kely, see \cite{MRS:1993}. Theorem \ref{thmatozero} and Theorem \ref{thmatoinfty} show that
the class of tests for $H_0$ are in a certain sense "closed at the boundaries" $a\to 0$ and $a \to \infty$.
However, in contrast to the test for multivariate normality based on $U_{n,a}$ for {\em fixed} $a \in (0,\infty)$, tests for $H_0$ based
on measures of multivariate skewness and kurtosis lack consistency against general alternatives, see, e.g.,
\cite{BHE:1991,BH:1992,HEN:1994a}.

\section{The limit null distribution of $U_{n,a}$}\label{seclimitnuuld}
In this section, we assume that the distribution of $X$ is some nondegenerate $d$-variate normal law. In view of affine invariance,
we may further assume that $\mathbb{E}(X) = 0$ and $\mathbb{E}(XX^\top) = {\rm I}_d$. By symmetry, it is readily seen that $U_{n,a}$
defined in \eqref{defuna} takes the form
\begin{equation}\label{reprquna}
U_{n,a} = \int_{\mathbb{R}^d} S_n^2(t) w_a(t) \, {\rm d}t,
\end{equation}
where
\begin{equation}\label{defprocsn}
S_n(t) = \frac{1}{\sqrt{n}} \sum_{j=1}^n \big{(} \|Y_{n,j}\|^2 + \|t\|^2 -d \big{)} \big{(} \cos(t^\top Y_{n,j})+ \sin(t^\top Y_{n,j})\big{)}, \quad t \in \mathbb{R}^d.
\end{equation}
In view of \eqref{reprquna}, our setting for asymptotics will be the separable Hilbert space $\mathbb{H}$ of (equivalence classes of)
measurable functions $f: \mathbb{R}^d \rightarrow \mathbb{R}$ that satisfy $\int f^2(t) w_a(t) \, {\rm d}t < \infty$. Here and
in the sequel, each unspecified integral will be over $\mathbb{R}^d$. The scalar product and the norm in $\mathbb{H}$ are given by
$\langle f,g\rangle_\mathbb{H} = \int f(t)g(t) w_a(t)\, {\rm d}t$ and $\|f\|_\mathbb{H} = \langle f,f\rangle_\mathbb{H}^{1/2}$, respectively. Notice that, in this notation,
\eqref{reprquna} takes the form $U_{n,a} = \|S_n\|^2_\mathbb{H}$, where $S_n$ is given in  \eqref{defprocsn}.

Putting $\psi(t) = \exp(-\|t\|^2/2)$ as before, and writing $\vertk$ for convergence in distribution, the main result of this section is as follows.

\begin{theorem} \label{thmlimitnull}
If $X$ has some nondegenerate normal distribution, we have the following:
\begin{enumerate}
\item[a)] There is a centered Gaussian random element ${\cal S}$ of $\mathbb{H}$ having covariance kernel
\begin{eqnarray*}
    K(s, t) &= & \psi(s-t)\Big{\{} 2d + \|s\|^2 \|t\|^2 - 2 s^\top t \|s-t\|^2 -4 \|s-t\|^2\Big{\}}\\
    & & \quad + 2 \psi(s)\psi(t) \Big{\{} 2\|s\|^2 + 2 \|t\|^2 - d - 2 s^\top t - 4 (s^\top t)^2\Big{\}}, \quad s,t \in \mathbb{R}^d,
    \end{eqnarray*}
such that, with $S_n$ defined in \eqref{defprocsn}, $S_n \vertk {\cal S}$ as $n \to \infty$.
\item[b)] We have
\begin{equation}\label{uinftya}
U_{n,a} \vertk \int {\cal S}^2(t) w_a(t) \, {\rm d}t \quad \text{as } n \to \infty.
\end{equation}
\end{enumerate}
\end{theorem}

\begin{proof} Since the proof is analogous to the proof of Proposition 3.2 of \cite{DEH:2019}, it will only be sketched.
If $S_n^0(t)$ stands for the modification of $S_n(t)$ that results if we replace $Y_{n,j}$ with $X_j$, then a Hilbert space
central limit theorem holds for $S_n^0$, since the summands of $S_n^0$ are square-integrable centered random elements of $\mathbb{H}$.
The idea is thus to find a random element $\widetilde{S}_n$ of $\mathbb{H}$ such that $\widetilde{S}_n \vertk {\cal S}$ and
$\|S_n - \widetilde{S}_n \|_\mathbb{H} = o_\mathbb{P}(1)$. Putting $Y_{n,j} = X_j + \Delta_{n,j}$ in
\eqref{defprocsn} and using the fact that $\cos(t^\top Y_{n,j}) = \cos(t^\top X_j) - \sin(\Theta_j)t^\top \Delta_{n,j}$,
$\sin(t^\top Y_{n,j}) = \sin(t^\top X_j) + \cos(\Gamma_j) t^\top \Delta_{n,j}$, where $\Theta_j, \Gamma_j$ depend on $X_1,\ldots,X_n$
 and $t$ and satisfy $|\Theta_j - t^\top X_j| \le |t^\top \Delta_{n,j}|$,   $|\Gamma_j - t^\top X_j| \le |t^\top \Delta_{n,j}|$,
 some algebra and Proposition A.1 of \cite{DEH:2019} show that a choice of $\widetilde{S}_n$ is given by
 \[
 \widetilde{S}_n(t) = \frac{1}{\sqrt{n}} \sum_{j=1}^n h(X_j,t),
 \]
 where
\begin{eqnarray*}
   h(x,t) & = & \bigl(\|x\|^2 + \|t\|^2 - d\bigr)\bigl(\cos(t^{\top}x) + \sin(t^{\top}x)\bigr) \\
    & & \quad - \psi(t)\Big{\{} 2\|t\|^2 + \|x\|^2 - d + 2t^\top x - 2(t^\top x)^2\Big{\}}.
    \end{eqnarray*}
    Tedious calculations then show that the covariance kernel of ${\cal S}$, which is $\mathbb{E}[h(X,s)h(X,t)]$, is equal to
    $K(s,t)$ given above. \end{proof}

Let $U_{\infty,a}$ denote a random variable having the limit distribution of $U_{n,a}$ given in \eqref{uinftya}. Since the distribution of $U_{\infty,a}$ is that of
$\|{\cal S}\|_\mathbb{H}^2$, where ${\cal S}$ is the Gaussian random element of $\mathbb{H}$ figuring in Theorem \ref{thmlimitnull}, it is the distribution of
$\sum_{j \ge 1} \lambda_jN_j^2$, where $N_1,N_2, \ldots $ is a sequence of i.i.d. standard normal random variables, and $\lambda_1,\lambda_2, \ldots $ are the positive eigenvalues
corresponding to normalized eigenfunctions of the integral operator $f \mapsto Af$ on $\mathbb{H}$,  where
$(Af)(s) = \int K(s,t) f(t) \, w_a(t) \, {\rm d}t$. It seems to be hopeless to obtain closed-form expressions of these eigenvalues. However, in view
of Fubini's theorem,
we have
\[
\mathbb{E}[U_{\infty,a}] = \int \mathbb{E}\big{[}{\cal S}^2(t)\big{]} w_a(t) \, {\rm d}t = \int K(t,t) w_a(t) \, {\rm d}t,
\]
and thus straightforward manipulations of integrals yield the following result.

\begin{theorem}
Putting $\gamma = (a/(a+1))^{d/2}$, we have
\[
  \mathbb{E}[U_{\infty,a}] = 2d \left(\frac{\pi}{a}\right)^{d/2} \bigg{\{} 1 - \gamma + \frac{\gamma}{a+1} - \frac{(d+2)\gamma}{(a+1)^2} + \frac{d+2}{8a^2}\bigg{\}}.
\]
\end{theorem}

\noindent From this result, one readily obtains
\begin{equation}\label{inftyatonull}
\lim_{a \to 0}\bigg{[} \left(\frac{a}{\pi}\right)^{d/2} \mathbb{E}\big{[}U_{\infty,a}\big{]} - \frac{d(d+2)}{4a^2} \bigg{]} = 2d.
\end{equation}
It is interesting to compare this limit relation with \eqref{limesato0}. If the underlying distribution is standard normal, i.e.,
if $\mathbb{P}^X = {\rm N}_d(0,{\rm I}_d)$, we have $\mathbb{E}\|X\|^4 = 2d + d^2$. Now, writing $Y_{n,j} = X_j + \Delta_{n,j}$ and
using Proposition A.1 of \cite{DEH:2019}, the right hand side of \eqref{limesato0} turns out to converge in probability  to
$\mathbb{E}\|X\|^4 - d^2$ as $n \to \infty$, and this expectation is the right hand side of \eqref{inftyatonull}. Regarding the case $a \to \infty$,
the representation of $\mathbb{E}[U_{\infty,a}]$ easily yields
\[
\lim_{a \to \infty} \bigg{[} \frac{2a^{d/2+1}}{\pi^{d/2}} \mathbb{E}[U_{\infty,a}] \bigg{]} = 2d(d+2).
\]
This result corresponds to \eqref{limitmrz}, since, by Theorem 2.2 of \cite{H:1997}, the right hand side of \eqref{limitmrz}, after multiplication
with $n$, converges in distribution to $2(d+2)\chi^2_d$ as $n \to \infty$ if $\mathbb{P}^X = {\rm N}_d(0,{\rm I}_d)$. Here, $\chi^2_d$ is
a random variable having a chi square distribution with $d$ degrees of freedom.

\section{Limits of $U_{n,a}$ under alternatives}\label{seclimitsalt}
In this section we assume that $X,X_1,X_2, \ldots $ are i.i.d., and that $\mathbb{E}\|X\|^4 < \infty$. Moreover, let
$\mathbb{E}(X) =0$ and $\mathbb{E}(XX^\top) = {\rm I}_d$ in view of affine invariance, and recall the Laplace operator $\Delta$
from Section \ref{sec:Intro}. The characteristic function of $X$ will be denoted by $\psi(t) = \mathbb{E}[\exp({\rm i}t^\top X)]$, $t \in \mathbb{R}^d$.  Letting
\[
\psi^\pm(t) = \mathbb{E}[\cos(t^\top X)] \pm \mathbb{E}[\sin(t^\top X)], \quad t \in \mathbb{R}^d,
\]
we first present an almost sure limit for $n^{-1} U_{n,a}$.

\begin{theorem}\label{thmalt1}
We have
\[
\frac{U_{n,a}}{n} \fsk \Gamma_a := \int_{\mathbb{R}^d} z^2(t)  w_a(t) \, {\rm d}t = \|z\|^2_\mathbb{H},
\]
where
\begin{equation}\label{defzvont}
z(t) = -\Delta\psi^+(t) + (\|t\|^2 - d)\psi^+(t).
\end{equation}
\end{theorem}

\begin{proof} In what follows, we write ${\rm CS}^\pm(\xi) = \cos(\xi) \pm \sin(\xi)$, and we put $Y_j = Y_{n,j}$, $\Delta_j = \Delta_{n,j}$
for the sake of brevity. From \eqref{reprquna} and \eqref{defprocsn}, we have $n^{-1}U_{n,a} = \|V_n + W_n\|_\mathbb{H}^2$,
where
\[
V_n(t) = \frac{1}{n} \sum_{j=1}^n \|Y_j\|^2 {\rm CS}^+(t^\top Y_j), \quad W_n(t) = (\|t\|^2-d)\,  \frac{1}{n} \sum_{j=1}^n {\rm CS}^+(t^\top Y_j).
\]
Putting
\[
V_n^0(t) = \frac{1}{n}\sum_{j=1}^n \|X_j\|^2 {\rm CS}^+(t^\top X_j), \quad W_n^0(t) = (\|t\|^2-d) \, \frac{1}{n}\sum_{j=1}^n {\rm CS}^+(t^\top X_j),
\]
the strong law of large numbers in Hilbert spaces (see, e.g., Theorem 7.7.2 of \cite{HEU:2015}) yields $\|V_n^0+ W_n^0\|_\mathbb{H}^2 \fsk \Gamma_a$ as $n \to \infty$, and thus it suffices to prove
 $\|V_n+ W_n\|_\mathbb{H}^2 - \|V_n^0+ W_n^0\|_\mathbb{H}^2 \fsk 0$. From
 \[
 \|V_n+ W_n\|_\mathbb{H}^2 - \|V_n^0+ W_n^0\|_\mathbb{H}^2 = \big{\langle} V_n-V_n^0+W_n-W_n^0,V_n+W_n + V_n^0+W_n^0\big{\rangle}_\mathbb{H},
 \]
 the Cauchy--Schwarz inequality, the fact that $|V_n(t)| \le 2d$, $\max(|W_n(t)|,|W_n^0(t)|) \le 2(d+\|t\|^2)$, $|V_n^0(t)| \le 2n^{-1}\sum_{j=1}^n \|X_j\|^2$
 and Minkowski's inequality, it suffices to prove $\|V_n-V_n^0\|_\mathbb{H} \fsk 0$ and $\|W_n - W_n^0\|_\mathbb{H} \fsk 0$ as $n \to \infty$. As for
$W_n-W_n^0$, the inequalities $|\cos(t^\top Y_j)- \cos(t^\top X_j)| \le \|t\|\, \|\Delta_j\|$,
$|\sin(t^\top Y_j)- \sin(t^\top X_j)| \le \|t\|\, \|\Delta_j\|$ and the Cauchy--Schwarz inequality yield $|W_n(t)-W_n^0(t)| \le (\|t\|^2 + d) 2 \|t\| (n^{-1}\sum_{j=1}^n \|\Delta_j\|^2)^{1/2}$.
In view of Proposition A.1 b) of \cite{DEH:2019}, we have $\|W_n - W_n^0\|_\mathbb{H} \fsk 0$. Regarding $V_n-V_n^0$, we decompose this difference according to
\[
V_n(t) - V_n^0(t) = \frac{1}{n}\sum_{j=1}^n (\|Y_j\|^2\! -\! \|X_j\|^2){\rm CS}^+(t^\top Y_j) + \frac{1}{n}\sum_{j=1}^n \|X_j\|^2 \big{(}{\rm CS}^+(t^\top Y_j)\! - \! {\rm CS}^+(t^\top X_j)\big{)}.
\]
The squared norm in $\mathbb{H}$ of the second summand on the right hand side converges to zero almost surely, see the treatment of $U_{n,1}$ in the proof of Theorem 6.1 of
\cite{DEH:2019}.  The same holds for the first summand, since its modulus is bounded from above by $4\|t\| n^{-1}\sum_{j=1}^n \|\Delta_j\| + 2n^{-1}\sum_{j=1}^n \|\Delta_j\|^2$,
and the inequality $n^{-1}\sum_{j=1}^n \|\Delta_j\| \le (n^{-1}\sum_{j=1}^n \|\Delta_j\|^2)^{1/2}$, together with
Proposition A.1 b) of \cite{DEH:2019}, yield the assertion.
\end{proof}

Since, under the conditions of Theorem \ref{thmalt1},  $\Gamma_a$ is strictly positive if the underlying distribution does not belong to
${\cal N}_d$, $U_{n,a}$ converges almost surely to $\infty$ under such an alternative, and we have the following result.

\begin{corollary} The test which reject the hypothesis $H_0$ for large values of $U_{n,a}$ is consistent against
each fixed alternative satisfying $\mathbb{E}\|X\|^4 < \infty$.
\end{corollary}

The next result, which corresponds to Theorem 6.4 of \cite{DEH:2019}, shows that the (population) measure
of multivariate skewness in the sense of M\'{o}ri, Rohatgi and Sz\'{e}kely  emerges as the limit
of $\Gamma_a$, after a suitable scaling, as $a\to \infty$.

\begin{theorem} Under the condition $\mathbb{E}\|X\|^6 < \infty$, we have
\[
\lim_{a\rightarrow\infty} 2a\left(\frac{a}{\pi}\right)^{d/2}\Gamma_a = \left\|\mathbb{E}\left(\|X\|^2X\right)\right\|^2.
\]
\end{theorem}

\begin{proof} By definition,
\begin{eqnarray*}
        \Gamma_a  \! & \! = \! & \! \int (\|t\|^2\! -\! d)^2\psi^+(t)^2 w_a(t)  {\rm d}t - 2\! \int\! (\|t\|^2\! -\!  d)\psi^+(t)\Delta \psi^+(t)w_a(t) \, {\rm d}t
            + \int \! (\Delta\psi^+(t))^2w_a(t)  {\rm d}t\\
             \! & \! = \! & \! \Gamma_{a,1} + \Gamma_{a,2} + \Gamma_{a,3} \text{ (say)}.
\end{eqnarray*}
In what follows, let $Y,Z$ be independent copies of $X$. Since $\psi^+(t)^2 = \mathbb{E}[{\rm CS}^+(t^\top Y){\rm CS}^+(t^\top Z)]$, the
addition theorems for the cosine and the sine function and symmetry yield
\[
\Gamma_{a,1} = \mathbb{E}\bigg{[}\int (\|t\|^2 - d)^2\cos\bigl(t^{\top}(Y - Z)\bigr)w_a(t) \, {\rm d}t\bigg{]}.
\]
Putting $c= Y-Z$, display \eqref{imprel} then gives
\begin{align*}
    \Gamma_{a,1} &= \left(\frac{\pi}{a}\right)^{d/2}\frac{1}{16a^4}\mathbb{E}\biggl[\biggl(16d^2a^3(a-1) + 4d(d+2)a^2 + \|Y - Z\|^4 \\
    &\hspace{3.2cm}+ (8da^2 - 4(d+2)a)\|Y - Z\|^2\biggr)\exp\left(-\frac{\|Y - Z\|^2}{4a} \right)\biggr].
\end{align*}
Likewise, it follows that $\psi^+(t)\Delta \psi^+(t) = - \mathbb{E}[\|Y\|^2 \cos(t^\top(Y-Z))]$, whence
\begin{eqnarray*}
    \Gamma_{a,2} & = & 2\mathbb{E}\bigg{[} \|Y\|^2\int (\|t\|^2 - d)\cos\bigl(t^{\top}(Y - Z)\bigr)w_a(t) \, {\rm d}t\bigg{]} \\
    & = & -2\left(\frac{\pi}{a}\right)^{d/2}\mathbb{E}\biggl[\|Y\|^2\left(\frac{\|Y - Z\|^2}{4a^2} + d - \frac{d}{2a}\right)\exp\left(-\frac{\|Y - Z\|^2}{4a}\right)\biggr].
\end{eqnarray*}

\noindent Finally,
\[
\Gamma_{a,3} = \left(\frac{\pi}{a}\right)^{d/2}\mathbb{E}\left[\|Y\|^2\|Z\|^2\exp\left(-\frac{\|Y - Z\|^2}{4a}\right)\right],
\]
and it follows that
\begin{align*}
    2a\left(\frac{a}{\pi}\right)^{d/2}\Gamma_a &= 2a\mathbb{E}\biggl[\|Y\|^2\|Z\|^2\exp\left(-\frac{\|Y - Z\|^2}{4a}\right)\biggr] \\
    &-4a\mathbb{E}\biggl[\|Y\|^2\left(\frac{\|Y - Z\|^2}{4a^2} + d - \frac{d}{2a}\right)\exp\left(-\frac{\|Y - Z\|^2}{4a}\right)\biggr] \\
    &+ \frac{1}{8a^3}\mathbb{E}\biggl[\biggl(16d^2a^3(a-1) + 4d(d+2)a^2 + \|Y - Z\|^4 \\
    &\hspace{2cm}+ (8da^2 - 4(d+2)a)\|Y - Z\|^2\biggr)\exp\left(-\frac{\|Y - Z\|^2}{4a}\right)\biggr].
\end{align*}
Now,  dominated convergence yields
\begin{eqnarray*}
    2a\left(\frac{a}{\pi}\right)^{d/2}\Gamma_a & = &  2ad^2 - \frac{1}{2} \mathbb{E}\big{[} \|Y\|^2 \|Z\|^2 \|Y-Z\|^2\big{]} - 4ad^2 + d \mathbb{E} \big{[} \|Y\|^2 \|Y-Z\|^2\big{]}\\
    & & \    + 2d^2 + 2d^2(a-1) -d^3 + o(1)
    \end{eqnarray*}
    as $a \to \infty$. Since $\mathbb{E}\|Y\|^2 = d = \mathbb{E}\|Z\|^2$ and $\mathbb{E}(Y) = \mathbb{E}(Z) =0$, we have
   \[
\mathbb{E}\big{[} \|Y\|^2 \|Z\|^2 \|Y-Z\|^2\big{]} = 2d \mathbb{E}\|Y\|^4 - 2 \mathbb{E}\big{\|}\|X\|^2 X\big{\|}^2, \quad
\mathbb{E} \big{[} \|Y\|^2 \|Y-Z\|^2\big{]} = \mathbb{E}\|Y\|^4 + d^2,
   \]
    and the assertion follows.
\end{proof}

We close this section with a result on the asymptotic normality of $U_{n,a}$ under fixed alternatives.
 That such a result holds in principle
   follows from  Theorem 1 of \cite{BEH:2017}. To state the main idea, write again ${\rm CS}^\pm(\xi) = \cos(\xi) \pm \sin(\xi)$
   and notice that, by \eqref{reprquna}, $U_{n,a} = \|{\cal S}_n\|_\mathbb{H}^2$, where ${\cal S}_n(t)$ is given in
   \eqref{defprocsn}. Putting
   \[
   {\cal S}_n^*(t) = \frac{{\cal S}_n(t) }{\sqrt{n}} = \frac{1}{n}\sum_{j=1}^n
   \big{(} \|Y_{n,j}\|^2 + \|t\|^2 -d \big{)} {\rm CS}^+(t^\top Y_{n,j}), \quad t \in \mathbb{R}^d,
   \]
Theorem \ref{thmalt1} and \eqref{defzvont} show that
\begin{eqnarray}\nonumber
\sqrt{n}\left(\frac{U_{n,a}}{n}- \Gamma_a\right) & = & \sqrt{n}\big{(} \|{\cal S}_n^*\|_\mathbb{H}^2 - \|z\|^2 \big{)} = \sqrt{n}\langle
{\cal S}_n^*-z,{\cal S}_n^*+z \rangle_\mathbb{H}\\ \nonumber
& = & \sqrt{n} \langle {\cal S}_n^*-z, 2z + {\cal S}_n^*-z \rangle_\mathbb{H}\\ \label{vnstern}
& = & 2 \langle {\cal V}_n^* ,z\rangle_\mathbb{H} + \frac{1}{\sqrt{n}} \|{\cal V}_n^*\|_\mathbb{H}^2,
\end{eqnarray}
where ${\cal V}_n^*(t) = \sqrt{n}({\cal S}_n^*(t)-z(t))$, $t \in \mathbb{R}^d$. In the sequel, let $\nabla(f)(t)$ denote the gradient of a differentiable
function $f:\mathbb{R}^d \rightarrow \mathbb{R}$, evaluated at $t$, and write
${\rm H}f(t)$ for the Hessian matrix of $f$ at $t$ if $f$ is  twice continuously differentiable.
By proceeding as in the proof of Theorem 3.3 of \cite{DEH:2019}, there is a centered Gaussian element ${\cal V}^*$ of $\mathbb{H}$ having covariance kernel
\[
K^*(s,t)  =  \mathbb{E}\big{[} h^*(X,s)h^*(X,t)\big{]}, \qquad s,t \in \mathbb{R}^d,
\]
where
\begin{eqnarray*}
h^*(x,t) \! & \!  = \! & \! \big{(}\|x\|^2 \! + \!  \|t\|^2\! -\! d\big{)} {\rm CS}^+(x,t)
+ \left(\frac{1}{2}\nabla\Delta\psi^+(t)^{\top}\! -\!  \frac{1}{2}(\|t\|^2\! -\! d)\nabla\psi^+(t)^{\top}\! \right) \!(xx^{\top}\! -\!  {\rm I}_d)t \\
    \! & \!  \! & \!  \quad + 2\nabla\psi^-(t)^{\top}x \! +\!  \bigl(\Delta\psi^-(t) \! - \!  (\|t\|^2 - d)\psi^-(t)\bigr)t^{\top}x
     \! + \! x^{\top} {\rm H} \psi^+(t)x \! -\!  (\|t\|^2-d)\psi^+(t),
\end{eqnarray*}
such that ${\cal V}_n^* \vertk {\cal V}^*$ as $n \to \infty$. In view of \eqref{vnstern} and the fact that the distribution of $2\langle {\cal V}^*,z\rangle_\mathbb{H}$ is
centered normal, we have the following result.

\begin{theorem}\label{thmaltnormal}
Under the standing assumptions stated at the beginning of this section, we have
\[
\sqrt{n}\left(\frac{U_{n,a}}{n}- \Gamma_a\right) \vertk {\rm N}(0,\sigma_a^2),
\]
where
\[
\sigma_a^2 = 4 \int_{\mathbb{R}^d} \int_{\mathbb{R}^d} K^*(s,t) z(s)z(t) w_a(s)w_a(t) \, {\rm d}s {\rm d}t.
\]
\end{theorem}

We remark that a consistent estimator of $\sigma_a^2$ can be obtained by analogy
with the reasoning given in \cite{DEH:2019}, see Lemma 6.6, Theorem 6.7 and Remark 6.11 of that paper.

\section{Simulations}\label{secsimul}
In this section, we present the results of a Monte Carlo simulation study on the finite-sample power of the tests based on $U_{n,a}$.
This study is twofold in the sense that we consider testing for both univariate and multivariate normality, where the latter case
is restricted to dimensions $d\in\{2,3,5\}$. Moreover, the study is designed to match and complement  the counterparts in \cite{DEH:2019},
Section 7, and \cite{HV:2019}, since we take exactly the same setting with regard to sample size, nominal level of significance
 and selected alternative distributions. In this way, we facilitate an easy comparison with existing procedures.
In the univariate case, we consider sample sizes $n\in\{20,50,100\}$ and restrict the simulations to $n\in\{20,50\}$ in the multivariate setting.
The nominal level of significance is fixed throughout all simulations to 0.05. We simulated empirical critical values under $H_0$
for $d^{-2}\left(a/\pi\right)^{d/2}U_{n,a}$ with 100~000 replications, see Table \ref{tab:cv}. The rows entitled '$\infty$' give
approximations of the quantiles of the limit random element $U_{\infty,a}=\int {\cal S}^2(t) w_a(t) \, {\rm d}t$ in Theorem \ref{thmlimitnull} b).
The entries have been  calculated by the method presented in \cite{DEH:2019}, Section 7, setting $\ell = 100~000$ and $m=2~000$
for $d\in\{2,3,5,10\}$. Note that this approach only relies on the structure of the covariance kernel given in Theorem \ref{thmlimitnull} a),
the multivariate normal distribution,  and the weight function.

\begin{table}[t]
\centering
\begin{tabular}{cc|rrrrrrr}

$d$ & $n\backslash a$ & 0.1 & 0.25 & 0.5 & 1 & 2 & 3 & 5 \\
 \hline
 \multirow{4}{*}{1}& 20 & 147.99 & 25.86 & 7.14 & 2.46 & 1.47 & 1.27 & 1.02 \\
  & 50 & 149.69 & 26.20 & 7.29 & 2.62 & 1.61 & 1.39 & 1.13 \\
  & 100 & 150.85 & 26.45 & 7.34 & 2.65 & 1.63 & 1.42 & 1.16 \\
  & $\infty$ & 152.52 & 27.70 & 7.94 & 2.43 & 1.61 & 1.43 & 1.16  \\
\hline
\multirow{4}{*}{2}& 20 & 64.59 & 11.61 & 3.41 & 1.26 & 0.72 & 0.61 & 0.49 \\
  & 50& 65.63 & 11.87 & 3.50 & 1.33 & 0.79 & 0.68 & 0.56 \\
  & 100 & 65.81 & 11.94 & 3.52 & 1.34 & 0.80 & 0.70 & 0.58 \\
  & $\infty$ & 66.33 & 12.12 & 3.46 & 1.39 & 0.78 & 0.71 & 0.58  \\
   \hline
\multirow{4}{*}{3}& 20 & 46.49 & 8.22 & 2.45 & 0.91 & 0.49 & 0.40 & 0.33 \\
  & 50 & 46.81 & 8.37 & 2.52 & 0.97 & 0.55 & 0.47 & 0.38 \\
  & 100 & 46.88 & 8.41 & 2.53 & 0.97 & 0.56 & 0.48 & 0.40 \\
  & $\infty$ & 51.69 &  8.38 & 2.55 & 0.92 & 0.55 & 0.48 & 0.41 \\
\hline
\multirow{4}{*}{5}& 20 & 35.79 & 6.11 & 1.79 & 0.65 & 0.31 & 0.25 & 0.20 \\
  & 50 & 36.05 & 6.20 & 1.85 & 0.70 & 0.37 & 0.30 & 0.25 \\
  & 100 & 36.07 & 6.23 & 1.86 & 0.71 & 0.38 & 0.31 & 0.26 \\
  & $\infty$ & 39.38 & 6.27 & 1.90 & 0.68 & 0.38 & 0.32 & 0.27\\
\hline
\multirow{4}{*}{10}& 20 & 30.13 & 4.93 & 1.33 & 0.42 & 0.17 & 0.12 & 0.09 \\
   & 50 & 30.21 & 5.01 & 1.41 & 0.49 & 0.23 & 0.18 & 0.14 \\
   & 100 & 30.22 & 5.02 & 1.42 & 0.50 & 0.25 & 0.19 & 0.15 \\
   & $\infty$ & 32.70 & 5.39 & 1.47 & 0.52 & 0.25 & 0.20 & 0.16 \\
\end{tabular}
\caption{Empirical quantiles for $d^{-2}\left(a/\pi\right)^{d/2}U_{n,a}$ and $\alpha=0.05$ (100~000 replications)}\label{tab:cv}
\end{table}

In the univariate case, we consider the following alternatives: symmetric distributions, like the Student t$_{\nu}$-distribution
with $\nu \in \{3, 5, 10\}$ degrees of freedom, as well as the uniform distribution U$(-\sqrt{3}, \sqrt{3})$,
and asymmetric distributions, such as the $\chi^2_{\nu}$-distribution with $\nu \in \{5, 15\}$ degrees of freedom,
the beta distributions B$(1, 4)$ and B$(2, 5)$, and the gamma distributions $\Gamma(1, 5)$ and $\Gamma(5, 1)$,
both parametrized by their shape and rate parameter, the Gumbel distribution Gum$(1, 2)$ with location parameter 1
and scale parameter 2, the Weibull distribution W$(1, 0.5)$ with scale parameter 1 and shape parameter 0.5, and the lognormal
distribution LN$(0, 1)$.  As representatives of bimodal distributions,  we simulate the mixture of  normal distributions
NMix$(p, \mu, \sigma^2)$, where the random variables are generated by $(1 - p) \, {\rm N}(0, 1) + p \, {\rm N}(\mu, \sigma^2)$,
$p \in (0, 1)$, $\mu \in \R$, $\sigma > 0$. Note that these alternatives can also be found in the simulation studies
presented in \cite{BE:2019,DEH:2019,RDC:2010}. We chose these alternatives in order to ease the comparison with many other
existing tests.

We contrast the results in Table \ref{tab:pow.U.1} with those of Table 4 in \cite{DEH:2019}, which exhibits powers of the
related test statistic $T_{n,a}$. First we oppose the tests $T_{n,a}$ and $U_{n,a}$. Remarkably,  the test based on $U_{n,a}$
shows a better performance for the NMix-alternatives, especially for the choice of the tuning parameter $a\in\{0.25,0.5\}$.
On the other hand, $U_{n,a}$ is almost uniformly dominated by $T_{n,a}$ for the t$_{\nu}$-distribution. If the underlying
distribution is $\chi^2$, beta, gamma, Weibull, Gumbel or lognormal, both procedures have a comparable power.
Table 4 in \cite{DEH:2019} also provides finite-sample powers of strong either time-honored or recent tests for
normality, like the Shapiro--Wilk test, the Shapiro--Francia test, the Anderson--Darling test,
the Baringhaus--Henze--Epps--Pulley test (BHEP), see \cite{HW:1997}, the del
Barrio--Cuesta-Albertos--M\'{a}tran--Rodr\'{i}guez-Rodr\'{i}guez test (BCMR), see \cite{DBCMRR:1999}, and the Betsch--Ebner test,
see \cite{BE:2019}. For a description of the test statistics and critical values, see \cite{DEH:2019} and the references therein.
A comparison shows that,  for  suitable choice of the tuning parameter,  $U_{n,a}$ can compete with each of these tests,
sometimes outperforming them, for example in case of the uniform distribution, $n=20$, and $a=0.25$, and the
$\chi^2_{15}$-distribution for all sample sizes and $a=5$, but mostly being on the same power level.
It is interesting to see that the finite-sample power of $U_{n,a}$ depends heavily on the choice of $a$. This observation is
in contrast to the behavior of $T_{n,a}$, the power of which depends much less on $a$.

In the multivariate case, the  alternative distributions are selected to match those employed in the simulation
 studies in \cite{DEH:2019,HV:2019},  and are given as follows. Let NMix$(p,\mu,\Sigma)$ be the normal mixture distribution
generated by $(1 - p) \, {\rm N}_d(0, {\rm I}_d) + p \, {\rm N}_d(\mu, \Sigma)$, where $p \in (0, 1)$, $\mu \in \R^d$,
and $\Sigma$ is a positive definite matrix. In this notation, $\mu=3$ stands for a $d$-variate vector of 3's,
and $\Sigma={\rm B}_d$ is a $(d \times d)$-matrix containing 1's on the main diagonal, and each
off-diagonal entry has the value $0.9$.  We denote by t$_\nu(0,{{\rm I}}_d)$ the multivariate t$_{\nu}$-distribution
with $\nu$ degrees of freedom, see \cite{AG:2009}. The acronym DIST$^d(\vartheta)$ stands for a $d$-variate random vector
with i.i.d. marginal laws that belong to the distribution DIST with parameter $\vartheta$. In the sequel,
DIST is either the Cauchy distribution C, the logistic distribution L, the gamma distribution $\Gamma$, or the
Pearson Type VII  distribution P$_{VII}$. For the latter distribution, $\vartheta$ denotes the number of
degrees of freedom. The spherical symmetric distributions have been simulated using the {\tt R} package
{\tt distrEllipse}, see \cite{RKSC:2006}. These are denoted by $\mathcal{S}^d(\mbox{DIST})$, where DIST stands
for the distribution of the radii, and was chosen to be the exponential, the beta and the $\chi^2$-distribution.

Tables \ref{tab:pow.U.2} - \ref{tab:pow.U.5} can be contrasted to Tables 5 - 7 in \cite{DEH:2019},
and  for $n=50$,  with Tables 3 - 5 in \cite{HV:2019}. Again, we start with a comparison of $T_{n,a}$ and $U_{n,a}$.
For $d=2$ (see Table \ref{tab:pow.U.2} and Table 5 in \cite{DEH:2019}), $T_{n,a}$ is outperformed by $U_{n,a}$
for NMix$(0.1,3,I_2)$ and  NMix$(0.9,3,B_2)$,  but shows a stronger performance for NMix$(0.5,3,B_2)$. In case of
the multivariate t$_\nu$-distributions,  both procedures have a similar performance, as well as for the
DIST$^d(\vartheta)$ distributions. The spherical symmetric distributions are dominated by $U_{n,a}$ for a suitable
choice of the tuning parameter, except for the $\mathcal{S}^d(\chi^2_{5})$ distribution, where a similar behaviour
is asserted. Again,  $U_{n,a}$ seems to be much more sensitive to the choice of a proper tuning parameter
than $T_{n,a}$. Competing tests of multivariate normality are the Henze--Visagie test, see \cite{HV:2019},
the Henze--Jim\'{e}nez-Gamero test, see \cite{HJG:2019}, the BHEP-test, the Henze--Jim\'{e}nez-Gamero--Meintanis
test, see \cite{HJM:2019}, and the energy test, see \cite{SR:2005}. A description of the test statistics,
as well as procedures for computing critical values is found in \cite{HV:2019}. The BHEP-test performs best for
the NMix$(0.1,3,I_2)$-distribution (NMIX1 in \cite{HV:2019}) but is outperformed by $T_{n,a}$ for NMix$(0.5,0,B_2)$,
and by $U_{n,a}$ for the NMix$(0.9,3,B_2)$ (NMIX2 in \cite{HV:2019}), where these procedures show the best performance
of all tests considered. A similar behavior is observed for the t$_\nu$- and the spherical symmetric distributions,
where again $U_{n,a}$ and $T_{n,a}$ are strong competitors to all procedures considered.

\section{A real data example}\label{secdata}
\begin{figure}[t]
\centering
\includegraphics[width=4.5cm]{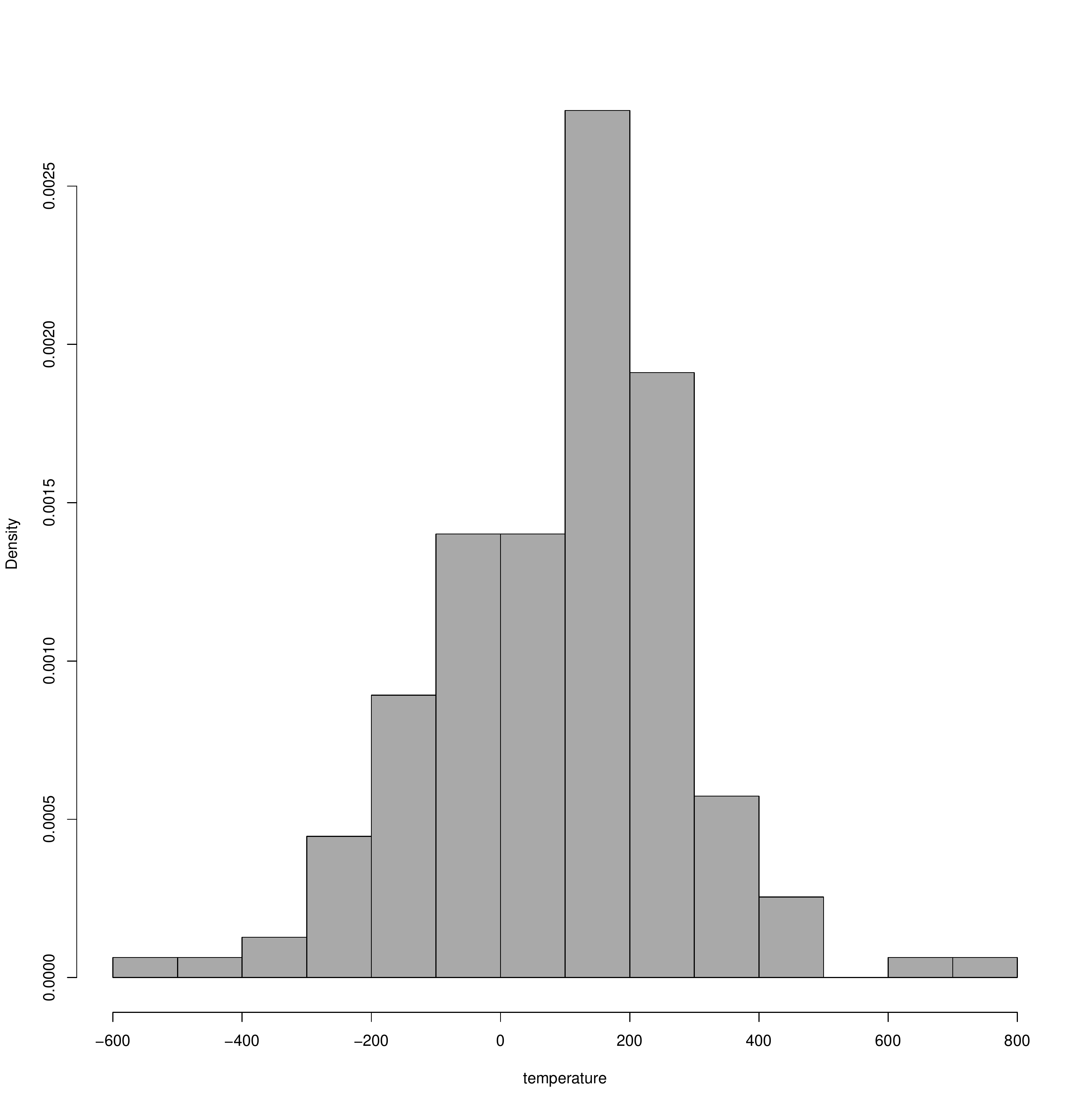}\hspace{0.5cm}\includegraphics[width=4.5cm]{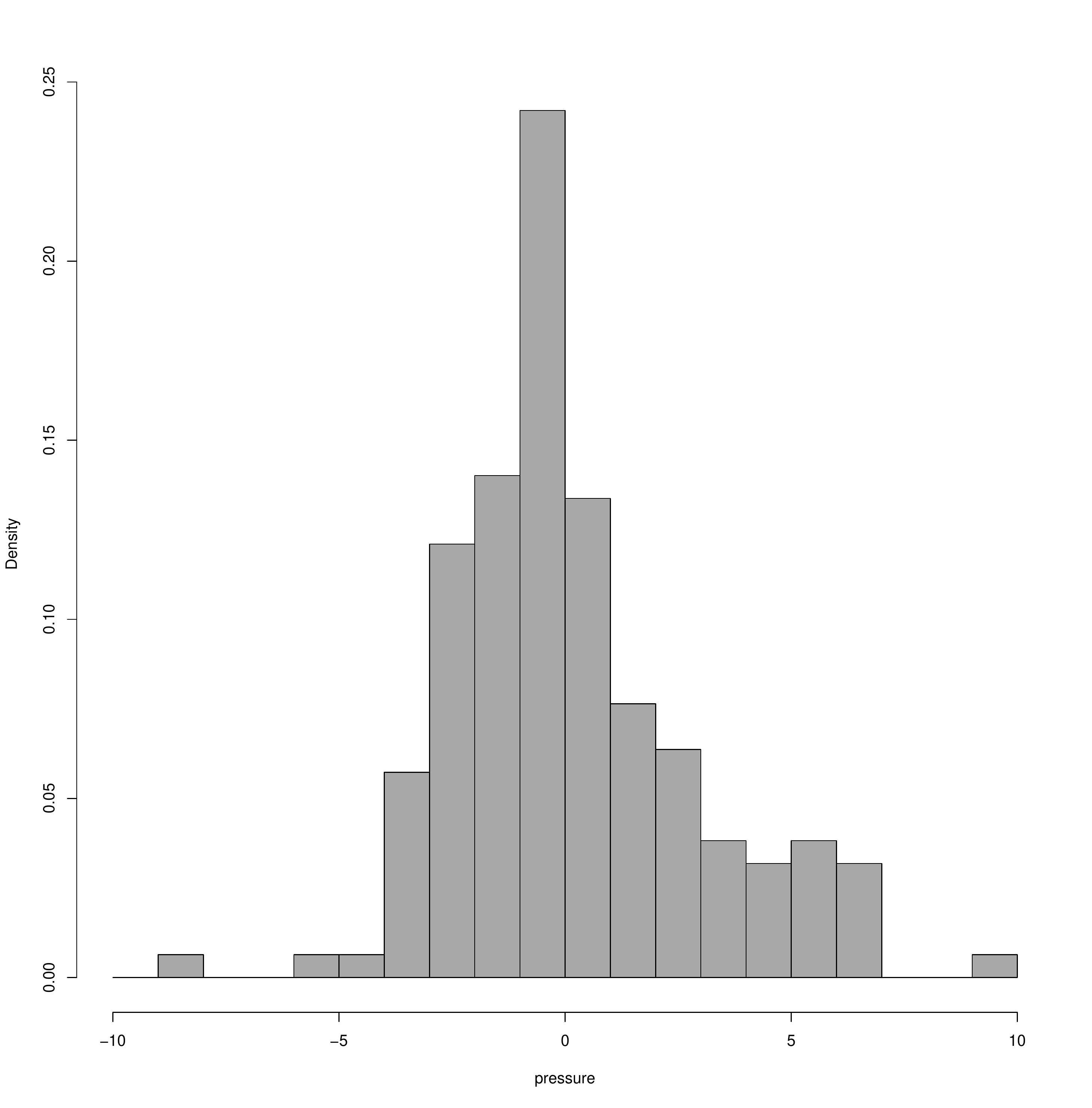}\hspace{0.5cm}\includegraphics[width=4.5cm]{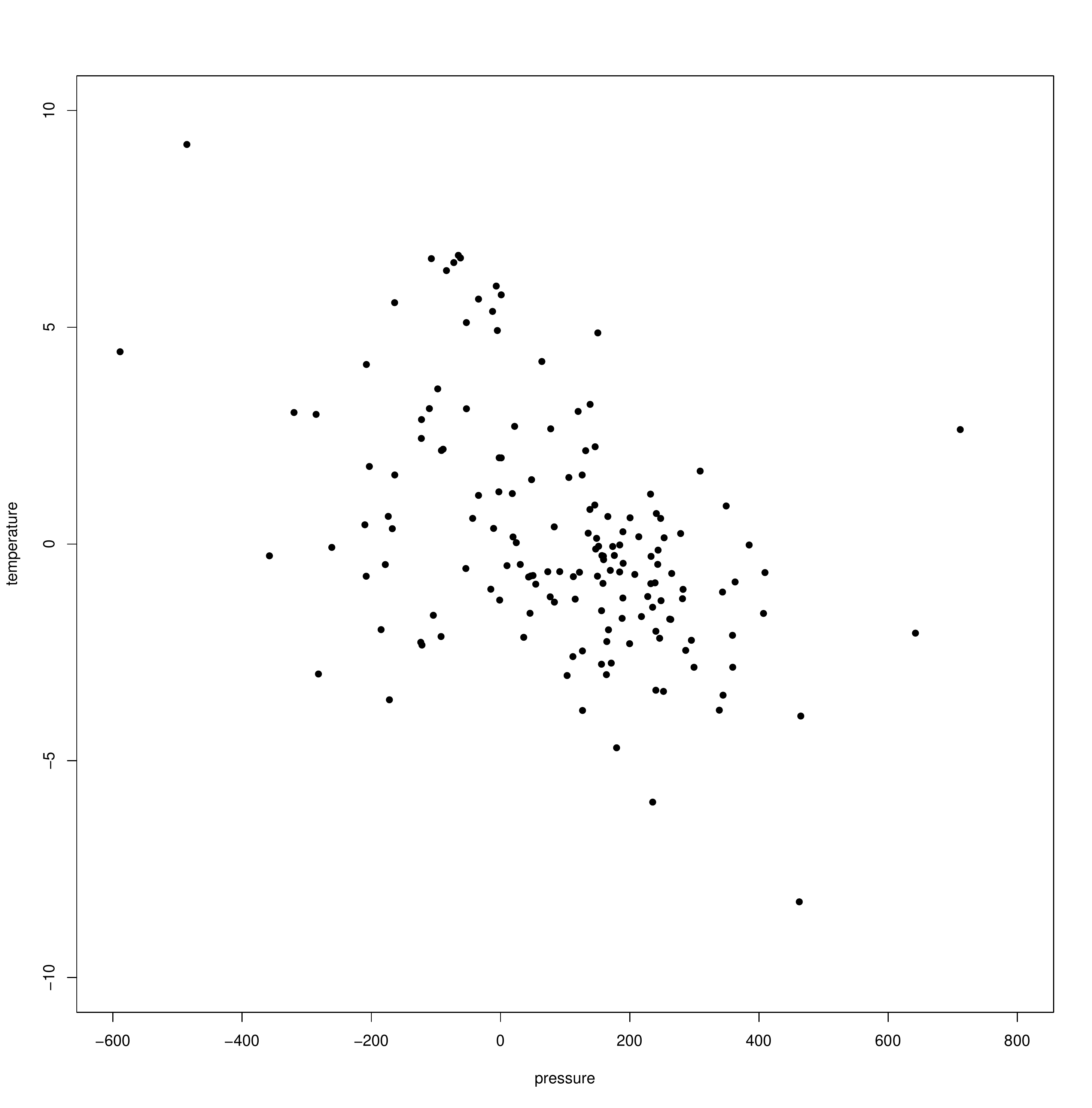} \\
\includegraphics[width=4.5cm]{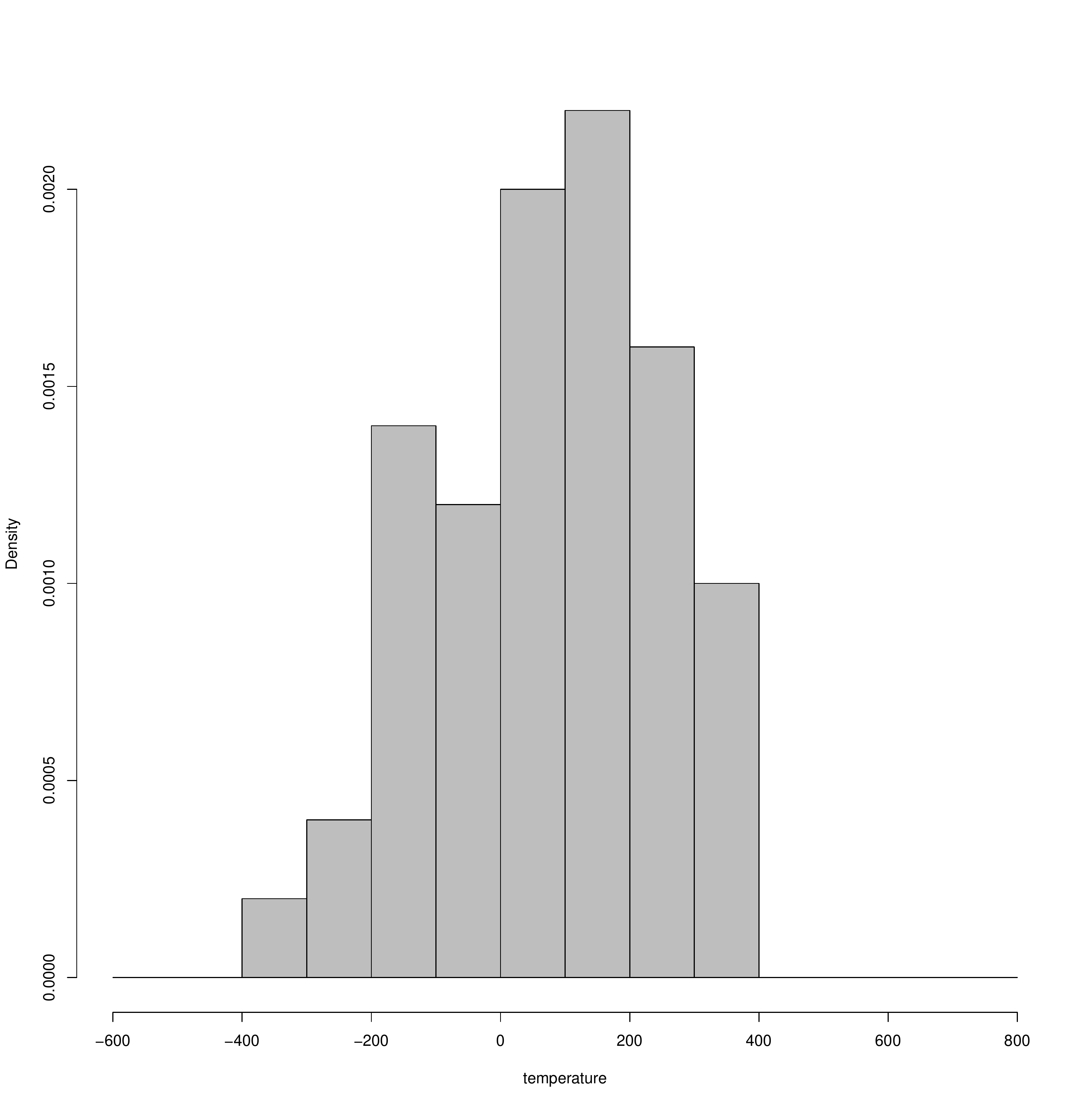}\hspace{0.5cm}\includegraphics[width=4.5cm]{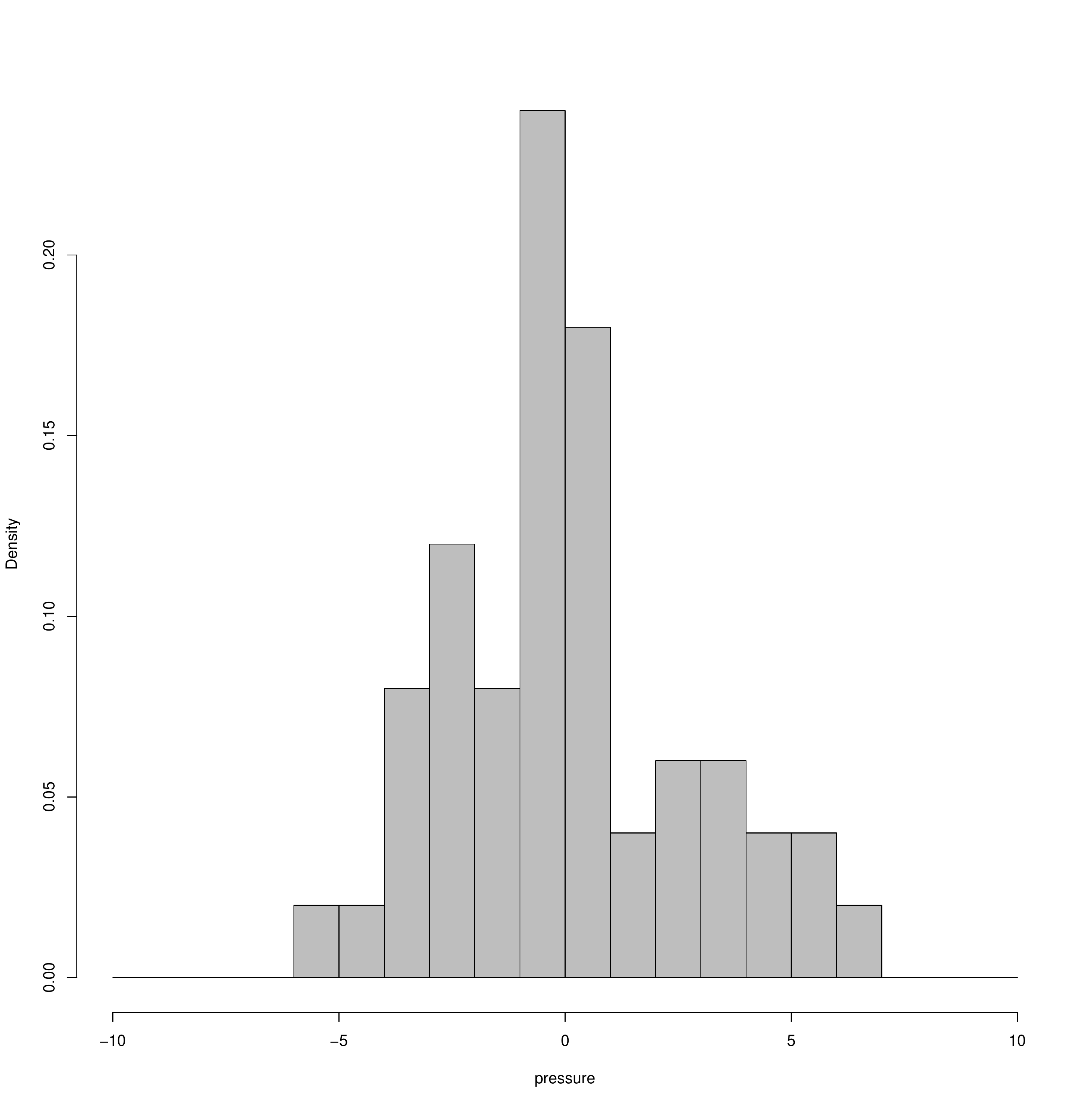}\hspace{0.5cm}\includegraphics[width=4.5cm]{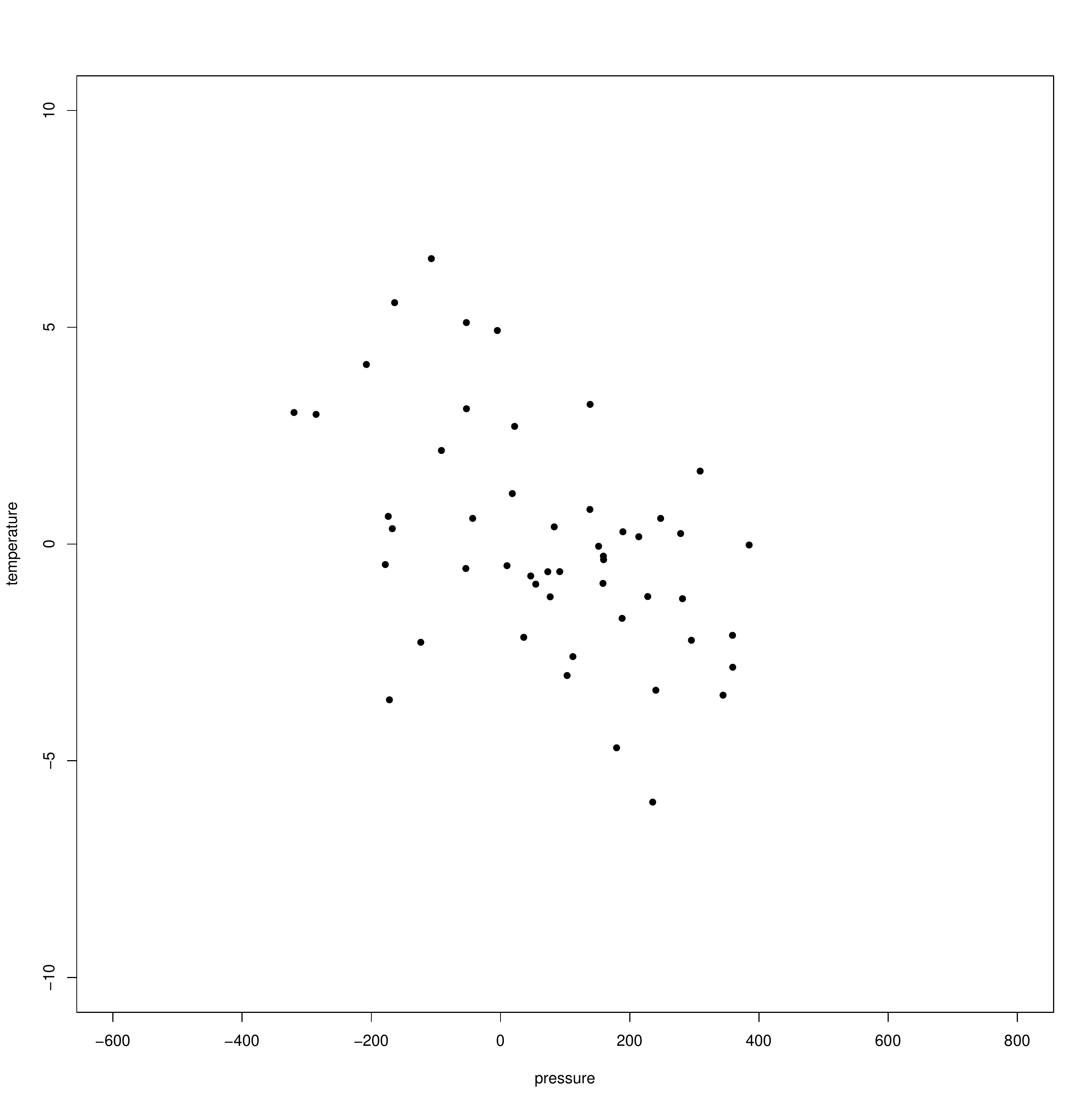}
\caption{Histogram of $n=157$ (upper row) and $n=50$ (lower row) differences between forecasts and observations of temperature (left) and pressure (middle) and scatterplot of temperature and pressure (right) in the North American Pacific Northwest.}\label{fig:weather}
\end{figure}
As a real data example,  we examine the meteorological data set \texttt{weather} provided in the \texttt{R} package \texttt{RandomFields},
see \cite{Setal:2019}, which consists of differences between forecasts and observations (forecasts minus observations) of
 temperature and pressure at $n=157$ locations in the North American Pacific Northwest.
 The data are pointwise realizations of a bivariate ($d=2$) error Gaussian random field, see Figure \ref{fig:weather}.
  The forecasts are from the GFS member of the University of Washington regional numerical weather prediction ensemble, see \cite{EM:2005},
   and they were valid on December 18, 2003 at 4 p.m. local time, at a forecast horizon of 48 hours.
   We ignore the given location of measurements in this evaluation and test the hypothesis that each pair of differences
   can be modeled as an i.i.d. copies from a bivariate normal distribution. In Table \ref{tab:real_data},
 we calculate empirical $p$-values based on 10~000 replications for $U_{n,a}$ for the univariate differences of temperature and pressure,
 as well as for the bivariate data for the whole data set, $n=157$, and for a random selection of $n=50$ points (selected in \texttt{R}
 with function \texttt{sample()} and seed fixed to '0721').
 Regarding the complete data set, we reject the hypothesis of normality in nearly all cases
 on a $5\%$ level of significance, while on a $1\%$ level of significance we are not able to reject $H_0$ for the differences
 in temperature. However, for the pressure and the bivariate data the hypothesis of normality is nearly always rejected.
 These results are not surprising, since the \texttt{weather} data set is an example of influence of spatial correlation,
 which has to be carefully modeled. In \cite{GKS:2010},  a bivariate Gaussian random field is fitted to the data, taking the mentioned spatial correlation into account, for a visualization of the locations see Figure 3 in \cite{GKS:2010}. For the subsample of points we see that the structure vanishes,
  and we throughout do not reject the hypotheses. Here, we have only applied our method as a proof of principle.
\begin{table}[b]
\centering
\begin{tabular}{lr|ccccccc}
Diff.& $n\backslash a$ & 0.1 & 0.25 & 0.5 & 1 & 2 & 3 & 5 \\ \hline
temperature & &  0.1042 & 0.0155 & 0.0102 & 0.0235 & 0.0292 & 0.0322 & 0.0407 \\
pressure  & 157 & 0.0128 & 0 & 0 & 0.0001 & 0.0001 & 0 & 0.0001 \\ \hdashline
bivariate & & 0.0001 & 0 & 0 & 0 & 0 & 0 & 0\\\hline
temperature & &  0.9472 & 0.6847 & 0.3675 & 0.3144 & 0.3145 & 0.3168 & 0.3337 \\
pressure  & 50 & 0.1649 & 0.2019 & 0.1822 & 0.2282 & 0.2169 & 0.2101 & 0.2109 \\\hdashline
bivariate & & 0.8485 & 0.6694 & 0.5528 & 0.5413 & 0.3497 & 0.2998 & 0.2879
\end{tabular}
\caption{Empirical $p$-values for $U_{n,a}$ for univariate and bivariate cases of the complete data set $n=157$ and the subsample $n=50$ (10~000 replications)}\label{tab:real_data}
\end{table}

\section{Conclusions and outlook}\label{secconclus}
We have introduced and studied a new affine invariant class of tests for multivariate normality that is easy to apply and
consistent against general alternatives. Although consistency has only been proved under the condition $\mathbb{E}\|X\|^4 < \infty$,
the test should be "all the more consistent" if $\mathbb{E}\|X\|^4 = \infty$, and we conjecture that, as is the case for the BHEP-tests,
also the test based on $U_{n,a}$ is consistent against {\em each} nonnormal alternative distribution.
A further topic of research
would be to choose the tuning parameter $a$ in an adaptive way, similar to the bootstrap based univariate approaches in \cite{AS:2015} and \cite{T:2019}.
It would also be of interest to obtain more information on the limit null distribution of $U_{n,a}$.
We finish the outlook by pointing out that, with respect to  the references in the introduction regarding
 other procedures and distributions, a similar analysis can be performed,  and it is of theoretical and practical relevance
 to study the resulting statistics in order to assess the influence of the options of estimating or not estimating certain of the pertaining functions.

After a comparison of $U_{n,a}$ and $T_{n,a}$ from \cite{DEH:2019},  and in view of the results of the simulation study,
 we recommend to use $T_{n,a}$, since it seems to be  more robust with respect to the choice of the tuning parameter $a$.
 Nevertheless, $U_{n,a}$ is a strong competitor,  and with a suitable data driven procedure for the choice of $a$ at hand,
 $U_{n,a}$ may turn out to be a favorable choice over the most classical and recent tests of uni- and multivariate normality.

\bibliographystyle{abbrv}
\bibliography{references}  

\begin{table}[t]
\centering
\begin{tabular}{lc|rrrrrrrr}
 \multicolumn{2}{c}{} & \multicolumn{7}{c}{$U_{n,a}$} \\
 Alt. & $n\backslash a$ & 0.1 & 0.25 & 0.5 & 1 &  2 & 3 & 5\\
  \hline
\multirow{3}{*}{N$(0,1)$} & 20 & 5 & 5 & 5 & 5 & 5 & 5 & 5 \\
  & 50 & 5 & 5 & 5 & 5 & 5 & 5 & 5 \\
  & 100 & 5 & 5 & 5 & 5 & 5 & 5 & 5 \\ \hline
  \multirow{3}{*}{NMix$(0.3,1,0.25)$} & 20 & 18 & 28 & 27 & 20 & 19 & 19 & 19 \\
  & 50 & 45 & 65 & 61 & 49 & 46 & 45 & 43 \\
  & 100 & 79 & 93 & 90 & 80 & 77 & 76 & 73 \\
  \multirow{3}{*}{NMix$(0.5,1,4)$} & 20 & 20 & 40 & 44 & 39 & 36 & 35 & 33 \\
  & 50 & 53 & 83 & 84 & 71 & 64 & 60 & 54 \\
  & 100 & 88 & 99 & 99 & 95 & 91 & 87 & 79 \\
  \hline
  \multirow{3}{*}{t$_3$} & 20 & 12 & 24 & 37 & 39 & 38 & 37 & 36 \\
  & 50 & 21 & 49 & 66 & 68 & 65 & 63 & 61 \\
  & 100 & 36 & 75 & 88 & 89 & 87 & 85 & 82 \\
  \multirow{3}{*}{t$_5$} & 20 & 6 & 11 & 19 & 22 & 21 & 21 & 21 \\
  & 50 & 8 & 21 & 36 & 40 & 38 & 37 & 35 \\
  & 100 & 11 & 34 & 54 & 59 & 57 & 55 & 51 \\
  \multirow{3}{*}{t$_{10}$} & 20 & 5 & 7 & 10 & 11 & 11 & 11 & 11 \\
  & 50 & 6 & 8 & 15 & 18 & 17 & 17 & 16 \\
  & 100 & 6 & 10 & 21 & 26 & 25 & 24 & 23 \\
  \hline
  \multirow{3}{*}{U$(-\sqrt{3},\sqrt{3})$} & 20 & 12 & 22 & 18 & 2 & 1 & 1 & 1 \\
  & 50 & 30 & 59 & 64 & 20 & 4 & 2 & 1 \\
  & 100 & 67 & 95 & 97 & 87 & 50 & 20 & 3 \\
  \hline
  \multirow{3}{*}{$\chi^2_5$} & 20 & 10 & 24 & 39 & 41 & 42 & 42 & 42 \\
  & 50 & 21 & 60 & 82 & 85 & 85 & 85 & 85 \\
  & 100 & 44 & 93 & 99 & 100 & 100 & 100 & 99 \\
  \multirow{3}{*}{$\chi^2_{15}$} & 20 & 5 & 9 & 16 & 18 & 19 & 19 & 19 \\
  & 50 & 7 & 18 & 37 & 44 & 45 & 45 & 46 \\
  & 100 & 9 & 34 & 64 & 75 & 76 & 76 & 77 \\
  \hline
  \multirow{3}{*}{B$(1,4)$} & 20 & 21 & 39 & 49 & 46 & 46 & 46 & 45 \\
  & 50 & 55 & 87 & 95 & 94 & 92 & 92 & 91 \\
  & 100 & 92 & 100 & 100 & 100 & 100 & 100 & 100 \\
  \multirow{3}{*}{B$(2,5)$} & 20 & 7 & 11 & 15 & 14 & 14 & 14 & 14 \\
  & 50 & 10 & 26 & 41 & 41 & 40 & 40 & 39 \\
  & 100 & 16 & 55 & 80 & 82 & 80 & 79 & 78 \\
  \hline
  \multirow{3}{*}{$\Gamma(1,5)$} & 20 & 39 & 64 & 75 & 73 & 73 & 73 & 72 \\
  & 50 & 86 & 99 & 100 & 100 & 100 & 100 & 99 \\
  & 1006 & 100 & 100 & 100 & 100 & 100 & 100 & 100 \\
  \multirow{3}{*}{$\Gamma(5,1)$} & 20 & 6 & 12 & 21 & 24 & 25 & 25 & 25 \\
  & 50 & 9 & 27 & 51 & 58 & 59 & 60 & 60 \\
  & 100 & 13 & 54 & 83 & 90 & 90 & 90 & 90 \\
  \hline		
  \multirow{3}{*}{W$(1,0.5)$} & 20 & 39 & 65 & 76 & 74 & 74 & 74 & 73 \\
  & 50 & 86 & 98 & 100 & 100 & 100 & 100 & 99 \\
  & 100 & 100 & 100 & 100 & 100 & 100 & 100 & 100 \\
  \multirow{3}{*}{Gum$(1,2)$} & 20 & 7 & 16 & 28 & 32 & 33 & 33 & 33 \\
  & 50 & 10 & 37 & 62 & 70 & 71 & 71 & 71 \\
  & 100 & 17 & 67 & 90 & 95 & 95 & 95 & 95 \\
  \multirow{3}{*}{LN$(0,1)$} & 20 & 62 & 82 & 90 & 89 & 89 & 89 & 89 \\
  & 50 & 97 & 100 & 100 & 100 & 100 & 100 & 100 \\
  & 100 & 100 & 100 & 100 & 100 & 100 & 100 & 100
\end{tabular}
\caption{Empirical power of $U_{n,a}$ ($d=1$, $\alpha = 0.05$, 10~000 replications)}\label{tab:pow.U.1}
\end{table}

\begin{table}[b]
\centering
\begin{tabular}{lc|rrrrrrrr}
\multicolumn{2}{c}{} & \multicolumn{7}{c}{$U_{n,a}$} \\
 Alt. & $n\backslash a$ & 0.1 & 0.25 & 0.5 & 1 &  2 & 3 & 5\\
  \hline
\multirow{2}{*}{N$_2(0,{\rm I}_2)$} & 20 & 5 & 5 & 5 & 5 & 5 & 5 & 5 \\
  & 50  & 5 & 5 & 5 & 5 & 5 & 5 & 5 \\
  \hline
  \multirow{2}{*}{NMix$(0.1,3,{\rm I}_2)$} & 20 & 8 & 18 & 38 & 40 & 39 & 39 & 38 \\
  & 50 & 12 & 48 & 83 & 88 & 87 & 86 & 85 \\
  \multirow{2}{*}{NMix$(0.5,0,{\rm B}_2)$} & 20 & 7 & 12 & 19 & 19 & 17 & 17 & 16 \\
  & 50 & 11 & 27 & 42 & 37 & 30 & 26 & 24 \\
  \multirow{2}{*}{NMix$(0.9,0,{\rm B}_2)$} & 20 & 7 & 12 & 23 & 26 & 25 & 24 & 24 \\
  & 50 & 8 & 23 & 48 & 54 & 52 & 49 & 47 \\
  \hline
  \multirow{2}{*}{t$_3(0,{\rm I}_2)$}& 20 & 15 & 34 & 53 & 57 & 55 & 53 & 51 \\
  & 50 & 32 & 69 & 87 & 90 & 89 & 87 & 84 \\
  \multirow{2}{*}{t$_5(0,{\rm I}_2)$}& 20 & 8 & 15 & 29 & 34 & 33 & 32 & 30 \\
  & 50 & 11 & 31 & 56 & 63 & 62 & 59 & 56 \\
  \multirow{2}{*}{t$_{10}(0,{\rm I}_2)$} & 20 & 6 & 8 & 14 & 17 & 16 & 16 & 15 \\
  & 50 & 6 & 11 & 24 & 30 & 29 & 27 & 26 \\
  \hline
  \multirow{2}{*}{C$^2(0,1)$} & 20  & 86 & 95 & 98 & 97 & 96 & 96 & 95 \\
  & 50 & 100 & 100 & 100 & 100 & 100 & 100 & 100 \\
  \multirow{2}{*}{L$^2(0,1)$} & 20 & 6 & 8 & 14 & 17 & 16 & 16 & 15 \\
  & 50 & 7 & 12 & 25 & 31 & 30 & 28 & 26 \\
  \multirow{2}{*}{$\Gamma^2(0.5,1)$} & 20 & 88 & 98 & 99 & 97 & 97 & 97 & 96 \\
  & 50 & 100 & 100 & 100 & 100 & 100 & 100 & 100 \\
  \multirow{2}{*}{$\Gamma^2(5,1)$} & 20 & 7 & 13 & 23 & 25 & 26 & 26 & 27 \\
  & 50 & 8 & 28 & 59 & 67 & 68 & 69 & 68 \\
\multirow{2}{*}{P$_{VII}^2(10)$} & 20 & 6 & 7 & 12 & 13 & 13 & 13 & 12 \\
  & 50 & 6 & 9 & 18 & 24 & 23 & 21 & 20 \\
  \hline
  \multirow{2}{*}{$\mathcal{S}^2(\mbox{Exp}(1))$} & 20 & 61 & 77 & 82 & 79 & 76 & 72 & 68 \\
  & 50 & 96 & 99 & 100 & 99 & 99 & 98 & 96 \\
  \multirow{2}{*}{$\mathcal{S}^2({\rm B}(1,2))$}& 20 & 28 & 35 & 31 & 23 & 21 & 19 & 17 \\
  & 50 & 65 & 73 & 60 & 39 & 31 & 25 & 18 \\
  \multirow{2}{*}{$\mathcal{S}^2(\chi^2_5)$}& 20 & 7 & 10 & 20 & 23 & 21 & 21 & 20 \\
  & 50 & 8 & 18 & 37 & 43 & 41 & 38 & 34
\end{tabular}
\caption{Empirical power of $U_{n,a}$ ($d=2$, $\alpha = 0.05$, 10~000 replications)}\label{tab:pow.U.2}
\end{table}

\begin{table}[b]
\centering
\begin{tabular}{lc|rrrrrrrr}
\multicolumn{2}{c}{} & \multicolumn{7}{c}{$U_{n,a}$} \\
 Alt. & $n\backslash a$ & 0.1 & 0.25 & 0.5 & 1 &  2 & 3 & 5\\
  \hline
\multirow{2}{*}{N$_3(0,{\rm I}_3)$} & 20 & 5 & 5 & 5 & 5 & 5 & 5 & 5 \\
  &50  & 5 & 5 & 5 & 5 & 5 & 5 & 4 \\
  \hline
  \multirow{2}{*}{NMix$(0.1,3,{\rm I}_3)$} & 20 & 9 & 18 & 38 & 39 & 40 & 40 & 39 \\
  & 50 & 12 & 48 & 89 & 93 & 91 & 90 & 89 \\
  \multirow{2}{*}{NMix$(0.5,0,{\rm B}_3)$} & 20 & 10 & 23 & 38 & 35 & 32 & 29 & 27 \\
  & 50 & 20 & 61 & 81 & 72 & 62 & 55 & 46 \\
  \multirow{2}{*}{NMix$(0.9,0,{\rm B}_3)$} & 20 & 9 & 19 & 38 & 43 & 42 & 40 & 38 \\
  &50 & 13 & 44 & 74 & 81 & 79 & 77 & 75 \\
  \hline
  \multirow{2}{*}{t$_3(0,{\rm I}_3)$} & 20 & 19 & 44 & 67 & 70 & 68 & 66 & 63 \\
  &50 & 41 & 85 & 96 & 97 & 97 & 96 & 94 \\
  \multirow{2}{*}{t$_5(0,{\rm I}_3)$} &  20 & 9 & 20 & 40 & 44 & 43 & 40 & 38 \\
  &50 & 15 & 45 & 75 & 81 & 80 & 77 & 72 \\
  \multirow{2}{*}{t$_{10}(0,{\rm I}_3)$} & 20 & 6 & 10 & 18 & 21 & 20 & 19 & 18 \\
  &50 & 7 & 16 & 35 & 43 & 41 & 38 & 35 \\
  \hline
  \multirow{2}{*}{C$^3(0,1)$} & 20 & 89 & 98 & 99 & 99 & 99 & 98 & 98 \\
  &50 & 100 & 100 & 100 & 100 & 100 & 100 & 100 \\
  \multirow{2}{*}{L$^3(0,1)$} & 20 & 6 & 9 & 16 & 18 & 17 & 16 & 15 \\
  &50 & 7 & 13 & 30 & 36 & 34 & 30 & 27 \\
  \multirow{2}{*}{$\Gamma^3(0.5,1)$} & 20 & 85 & 97 & 99 & 98 & 97 & 97 & 97 \\
  &50 & 100 & 100 & 100 & 100 & 100 & 100 & 100 \\
  \multirow{2}{*}{$\Gamma^3(5,1)$} & 20 & 7 & 12 & 23 & 24 & 26 & 27 & 27 \\
  &50 & 9 & 28 & 60 & 68 & 71 & 71 & 71 \\
  \multirow{2}{*}{P$_{VII}^3(10)$} & 20 & 6 & 8 & 13 & 14 & 14 & 13 & 12 \\
  &50 & 7 & 10 & 22 & 28 & 27 & 24 & 22 \\
  \hline
  \multirow{2}{*}{$\mathcal{S}^3(\mbox{Exp}(1))$} & 20 & 86 & 95 & 97 & 96 & 95 & 93 & 89 \\
  &50 & 100 & 100 & 100 & 100 & 100 & 100 & 100 \\
  \multirow{2}{*}{$\mathcal{S}^3({\rm B}(1,2))$} & 20 & 60 & 73 & 74 & 66 & 61 & 55 & 48 \\
  &50 & 96 & 99 & 98 & 96 & 94 & 89 & 78 \\
  \multirow{2}{*}{$\mathcal{S}^3(\chi^2_5)$} & 20 & 14 & 31 & 50 & 51 & 49 & 45 & 41 \\
  &50 & 31 & 68 & 86 & 88 & 86 & 82 & 76 \\
\end{tabular}
\caption{Empirical power of $U_{n,a}$ ($d=3$, $\alpha = 0.05$, 10~000 replications)}\label{tab:pow.U.3}
\end{table}

\begin{table}[b]
\centering
\begin{tabular}{lc|rrrrrrrr}
\multicolumn{2}{c}{} & \multicolumn{7}{c}{$U_{n,a}$} \\
 Alt. & $n\backslash a$ & 0.1 & 0.25 & 0.5 & 1 &  2 & 3 & 5\\
 \hline
\multirow{2}{*}{N$_5(0,{\rm I}_5)$} & 20 & 5 & 5 & 5 & 5 & 5 & 5 & 5 \\
  &50 & 5 & 5 & 5 & 5 & 5 & 5 & 5 \\
  \hline
  \multirow{2}{*}{NMix$(0.1,3,{\rm I}_5)$} & 20 & 9 & 17 & 28 & 31 & 33 & 33 & 32 \\
  &50 & 13 & 37 & 74 & 79 & 84 & 85 & 85 \\
  \multirow{2}{*}{NMix$(0.5,0,{\rm B}_5)$}& 20 & 24 & 57 & 74 & 69 & 64 & 59 & 53 \\
  &50 & 47 & 97 & 100 & 98 & 97 & 94 & 89 \\
  \multirow{2}{*}{NMix$(0.9,0,{\rm B}_5)$}& 20 & 14 & 36 & 54 & 58 & 59 & 56 & 54 \\
  &50 & 43 & 86 & 94 & 95 & 95 & 94 & 93 \\
  \hline
  \multirow{2}{*}{t$_3(0,{\rm I}_5)$}& 20 & 29 & 67 & 83 & 84 & 82 & 79 & 75 \\
  &50 & 73 & 99 & 100 & 100 & 100 & 100 & 99 \\
   \multirow{2}{*}{t$_5(0,{\rm I}_5)$}& 20 & 15 & 37 & 56 & 59 & 56 & 52 & 48 \\
  &50 & 37 & 82 & 94 & 96 & 95 & 93 & 89 \\
  \multirow{2}{*}{t$_{10}(0,{\rm I}_5)$}& 20 & 9 & 16 & 26 & 28 & 26 & 24 & 22 \\
  &50 & 14 & 38 & 61 & 65 & 63 & 58 & 52 \\
  \hline
  \multirow{2}{*}{C$^5(0,1)$}& 20 & 92 & 100 & 100 & 100 & 100 & 100 & 99 \\
  &50 & 100 & 100 & 100 & 100 & 100 & 100 & 100 \\
  \multirow{2}{*}{L$^5(0,1)$}& 20 & 6 & 11 & 17 & 17 & 16 & 15 & 13 \\
  &50& 10 & 23 & 40 & 43 & 40 & 36 & 31 \\
  \multirow{2}{*}{$\Gamma^5(0.5,1)$}& 20 & 73 & 95 & 98 & 97 & 98 & 98 & 97 \\
  &50 & 100 & 100 & 100 & 100 & 100 & 100 & 100 \\
  \multirow{2}{*}{$\Gamma^5(5,1)$}& 20 & 8 & 12 & 18 & 21 & 23 & 24 & 23 \\
  &50 & 11 & 30 & 54 & 62 & 69 & 71 & 71 \\
  \multirow{2}{*}{P$_{VII}^5(10)$}& 20 & 6 & 9 & 14 & 15 & 14 & 13 & 12 \\
  &50 & 8 & 18 & 30 & 33 & 31 & 28 & 24 \\
  \hline
  \multirow{2}{*}{$\mathcal{S}^5(\mbox{Exp}(1))$}& 20 & 97 & 100 & 100 & 100 & 100 & 99 & 99 \\
  &50 & 100 & 100 & 100 & 100 & 100 & 100 & 100 \\
  \multirow{2}{*}{$\mathcal{S}^5({\rm B}(1,2))$}& 20 & 88 & 97 & 98 & 96 & 95 & 92 & 88 \\
  &50 & 100 & 100 & 100 & 100 & 100 & 100 & 100 \\
  \multirow{2}{*}{$\mathcal{S}^5(\chi^2_5)$}& 20 & 44 & 77 & 88 & 87 & 84 & 80 & 74 \\
  &50 & 87 & 100 & 100 & 100 & 100 & 100 & 99
\end{tabular}
\caption{Empirical power of $U_{n,a}$ ($d=5$, $\alpha = 0.05$, 10~000 replications)}\label{tab:pow.U.5}
\end{table}



\end{document}